\documentclass[12pt,reqno,fleqn]{amsart}  
\usepackage{tikz}
\usepackage{amsmath,amstext,amsthm,amssymb,amsxtra}
\usepackage[top=1.5in, bottom=1.5in, left=1.25in, right=1.25in]	{geometry}
\usepackage{lmodern}

\usepackage{graphics}

\usepackage{mathtools}
\mathtoolsset{showonlyrefs,showmanualtags}

\usepackage{hyperref} 
\hypersetup{
	colorlinks=true,       
	linkcolor=blue,          
	citecolor=magenta,        
	filecolor=magenta,      
	urlcolor=cyan           
}

\usetikzlibrary {positioning}
\definecolor {processblue}{cmyk}{0.96,0,0,0}

\usepackage[msc-links]{amsrefs}

\newtheorem{theorem}{Theorem}[section]
\newtheorem{problem}{Problem}

\newtheorem{lemma}{Lemma}[section]
\newtheorem{corollary}{Corollary}[section]

\theoremstyle{definition}
\newtheorem{definition}{Definition}[section]
\theoremstyle{definition}

\theoremstyle{remark}
\newtheorem{remark}{Remark}[section]

\def\sign{{\rm sign\,}}

\def\MM^d{\ensuremath{\mathfrak M}}
\def\MM{\ensuremath{\mathcal M}}

\def\ZB{\ensuremath{\mathcal B}}
\def\ZA{\ensuremath{\mathcal A}}

\def\zZ{\ensuremath{\mathcal Z}}

\def\ZN{\ensuremath{\mathbb N}}

\def\zI{\ensuremath{\mathcal I}}

\def\ZD{\ensuremath{\cal D}}

\def\ZD{\ensuremath{\mathcal D}}

\numberwithin{equation}{section}
\def\md#1#2\emd{\ifx0#1
	\begin{equation*} #2 \end{equation*}\fi  
	\ifx1#1\begin{equation}#2\end{equation}\fi   
	\ifx2#1\begin{align*}#2\end{align*}\fi   
	\ifx3#1\begin{align}#2\end{align}\fi    
	\ifx4#1\begin{gather*}#2\end{gather*}\fi  
	\ifx5#1\begin{gather}#2\end{gather}\fi   
	\ifx6#1\begin{multline*}#2\end{multline*}\fi  
	\ifx7#1\begin{multline}#2\end{multline}\fi  
	\ifx8#1\begin{multline*}\begin{split}#2\end{split}\end{multline*}\fi
	\ifx9#1\begin{multline}\begin{split}#2\end{split}\end{multline}\fi
}
\newcommand {\e }[1]{\eqref{#1}}
\newcommand {\lem }[1]{Lemma \ref{#1}}
\newcommand {\rem }[1]{Remark \ref{#1}}

\newcommand {\pro }[1]{Proposition \ref{#1}}
\newcommand {\trm }[1]{Theorem \ref{#1}}

\newcommand {\sect }[1]{Section \ref{#1}}

\title[] {Asymptotic estimates for double-coverings}
\author{Grigori A. Karagulyan}
\address{Institute of Mathematics of NAS RA, Marshal Baghramian ave., 24/5, Yerevan, 0019, Armenia} 
\email{g.karagulyan@gmail.com}
\author{Vahe G. Karagulyan}
\address{Faculty of Mathematics and Mechanics, Yerevan State
	University, Alex Manoogian, 1, 0025, Yerevan, Armenia} 
\email{vahekar2000@gmail.com}

\thanks{The work was supported by the Science Committee of RA, in the frames of the research project 21AG‐1A045 }

\subjclass[2010]{05A17, 05A20, 05B40, 42C05, 60G42}
\keywords{double-covering, asymptotic estimate, Rademacher chaos, unconditional convergence}
	
\begin{document}
\begin{abstract}
A collection of finite sets $\{A_1,A_2,\ldots,A_{p}\}$ is said to be a double-covering if each $a\in \cup_{k=1}^{p}A_k$ is included in exactly two sets of the collection. For fixed integers $l$ and $p$, let $\mu_{l,p}$ be the number of equivalency classes of double-coverings with $\#(A_k)=l$, $k=1,2,\ldots,p$. We characterize the asymptotic behavior of the quantity $\mu_{l,p}$ as $p\to \infty$. The results are applied to give an alternative approach to the Bonami-Kiener \cite{Bon,Kie} hypercontraction  inequality. 
\end{abstract}

	\maketitle  
\section{Introduction}
We will consider a combinatorial object of double-covering, which is partially related to the theory of partitions and the graph theory. The theory of partitions has a rich history that is thoroughly considered in the book by G.~Andrews \cite{And}. A partition of a positive integer $n$ is a finite nonincreasing sequence of positive integers $n_1, n_2, \ldots, n_m$ such that $\sum_{k=1}^m n_k=n$. The partition function $\nu(n)$ is the number of partitions of $n$. One of the crowning achievement of the theory of partition is the precise formula for  $\nu(n)$ mostly completed by G.~Hardy  and S.Ramanujan and fully completed by H.~Rademacher (see \cite{And}, Theorem 5.1). The formula doesn't have a simple form, it is given in the form of an infinite series, where certain trigonometric sums are involved. There is also an asymptotic formula 
\begin{equation}\label{part}
	\nu(n)\sim \frac{1}{4n\sqrt{3}}\exp\left(\pi\left(\frac{2n}{3}\right)^{1/2}\right)
\end{equation}
due to Meinardus (see \cite{And}, Theorem 6.2), which will be used below.

A given family of finite sets $A_k$, $k=1,2,\ldots,p$, of arbitrary nature (not necessarily different) defines a set-collection denoted by $\ZA=\{A_1,A_2,\ldots,A_{p}\}$. 
\begin{definition}
	Two set-collections $\ZA=\{A_1,A_2,\ldots,A_{p}\}$ and $\ZB=\{B_1,B_2,\ldots,B_{p}\}$ is said to be equal if there is a one to one mapping $\sigma$ on $\{1,2,\ldots,p\}$ such that $A_{\sigma(k)}=B_k$ for all $k=1,2,\ldots,p$. We say that the set-collections $\ZA$ and $\ZB$ are isomorphic if there is a one to one mapping 
	\begin{equation*}
		\phi:\bigcup_{k=1}^{p}A_k\to\bigcup_{k=1}^{p}B_k
	\end{equation*}
	such that $\{\phi(A_1),\ldots,\phi(A_{p})\}=\{B_1,\ldots,B_{p}\}$. The relations of equality and being isomorphic for two set-collections $\ZA$ and $\ZB$ will be denoted by $\ZA=\ZB$ and $\ZA\sim \ZB$ respectively.
\end{definition}
Note that the symbol $\sim$ here determines an equivalence relation. Let $[A_1,A_2,\ldots,A_{p}]$ be the equivalence class (E-class) generated by a set-collection $\{A_1,A_2,\ldots,A_{p}\}$. Correspondingly, if $\ZD$ is a family of set-collections, then we denote by $[\ZD]$ the E-classes generated by the elements of $\ZD$.

Given set-collections $\ZA_1,\ZA_2,\ldots,\ZA_n$ determine a new set-collection $\{\ZA_1,\ZA_2,\ldots,\ZA_n\}$, which is a collection of all the elements of the set-collections $\ZA_k$. 
\begin{definition}
	A set-collection $\ZA=\{A_1,A_2,\ldots,A_{p}\}$ is said to be double-covering if each $a\in \cup_{k=1}^{p}A_k$ is included in exactly two sets $A_k$. If each such $a$ is covered by even number of sets $A_k$, then we say $\ZA$ is an even-covering. 
\end{definition}
\begin{definition}
	A set-collection $\ZA$ is said to be connected if for any two different elements $A, B\in \ZA$ there is a sequence of sets $A=E_0, E_1,\ldots,E_m=B$
	such that $E_j\in \ZA$, $E_{j}\cap E_{j+1}\neq\varnothing $ for any $j=0,1\ldots, m-1$. 
\end{definition}
\begin{definition}
	It is clear that every set-collection $\ZA$ can be uniquely written in the form $\{\ZA_1,\ZA_2,\ldots,\ZA_n\}$, where each $\ZA_k$ is a connected set-collection. In that case we will say $\ZA_k$ is a connected component of $\ZA$. We say a set-collection $\ZA$ is standard if all E-classes of connected components $[\ZA_k]$ are different. 
\end{definition}

The notation $\#(A)$ will stand for the number of the elements of a finite set $A$. For given integers $l\ge 2$ and $p\ge 2$ we consider a class of double-coverings  
\begin{equation}\label{x26}
	\ZA=\{A_1,A_2,\ldots,A_{p}\},
\end{equation}
satisfying $\#(A_k)= l$. Clearly the number of different E-classes of each such a class of double-coverings is an integer depending only on $l$ and $p$. We denote this number by $\mu_{l,p}$. The number of E-classes of a standard double-covering  will be denoted by $\mu_{l,p}(\text{standard})$. For the number $\mu_{l,p}$ of the full family of double-coverings sometimes we will also use the notation $\mu_{l,p}(\text{full})$. We will consider also double-coverings \e{x26} with $\#(A_k)= l$. The number of such double-coverings and their standard types will be denoted by $\mu_{l,p}(\text{full})$ and $\mu_{l,p}(\text{standard})$ respectively. Similarly we define a wider class of double-coverings \e{x26}, where $A_k$ satisfy $1\le \#(A_k)\le l$. We call first and the second kind of coverings $(l,p)$ and $(l,p)^*$ double-coverings respectively. For the number of E-classes of $(l,p)^*$-double-coverings we will use the notation $\mu_{l,p}^*(\text{full})$.

The quantities $\mu_{l,p}$ and $\mu_{l,p}^*$ can be also considered as objects in the multigraph theory. Namely, $\mu_{l,p}$ ($\mu_{l,p}^*$) is the number of $p$-vertex unlabeled multigraphs with $l$-degree ($\le l$-degree) vertexes. Indeed, to each double-covering \e{x26} we associate a multigraph with vertexes $A_k$, where the number of edges connected two different vertexes $A_k$ and $A_j$ is equal to the number of elements in the intersection set $A_k\cap A_j$. In fact, the elements of the union $\cup_j A_j$ get role of edges in this graph.
Here is an example of $(3,6)$double-covering,
\begin{align*}
	A_1=\{1,2,3\},\\
	A_2=\{3,4,5\},\\
	A_3=\{4,5,6\},\\
	A_4=\{6,7,8\},\\
	A_5=\{7,8,9\},\\
	A_6=\{1,2,9\},
\end{align*}
with its graph in Fig.\ref{f1}.
\begin{figure}\label{f1}
\begin {center}
\begin {tikzpicture}[-latex ,auto ,node distance =2 cm and 3cm ,on grid ,
semithick ,
state/.style ={ circle ,top color =white , bottom color = processblue!20 ,
	draw,processblue , text=blue , minimum width =0.1cm}]
\node[state] (A){$A_1$};
\node[state] (B) [right= 2.5cm of A] {$A_2$};
\node[state] (C) [right= 2.5cm of B] {$A_3$};
\node[state] (D) [right= 2.5cm of C] {$A_4$};
\node[state] (E) [right= 2.5cm of D] {$A_5$};
\node[state] (F) [right= 2.5cm of  E ] {$A_6$};
\path (A) edge  [-]  [bend left =30] node[below =0.05 cm] {$1$} (F);
\path (A) edge  [-] [bend right =30] node[above =0.05 cm] {$2$} (F);
\path (B) edge  [-] [bend right =25] node[below =0.05 cm] {$4$} (C);
\path (B) edge  [-] [bend left=25] node[above =0.05 cm] {$5$} (C);
\path (C) edge  [-] node[below =0.05 cm] {$6$} (D);
\path (D) edge [-] [bend right =25] node[below =0.05 cm]  {$7$}(E);
\path (D) edge [-] [bend left =25] node[above =0.05 cm]  {$8$}(E);
\path (A) edge  [-]  node[below =0.05 cm] {$3$} (B);
\path (E) edge  [-]  node[below =0.05 cm] {$9$} (F);
\end{tikzpicture}
\end{center}
\caption{$(3,6)$-multi-graph.}
\end{figure}
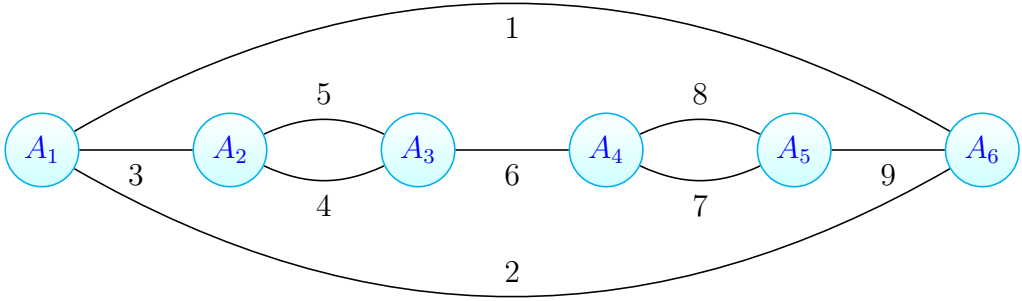

Double coverings with the parameter $l=2$ have a simple characterization in contrast to the case of $l>2$. Indeed, first observe that all the connected $(2,p)$-double-coverings \e{x26} are isomorphic to the cyclic set-collection $\{\{1,2\}, \{2,3\},\ldots,\{p-1,p\},\{p,1\}\}$. If a double-covering \e{x26} is arbitrary, then we have $\ZA=\{\ZA_1,\ZA_2,\ldots,\ZA_n\}$, where $\ZA_k$ are the connected components of $\ZA$. If $p_k=\#(\ZA_k)$, then we have  $p_k\ge 2$ and $p_1+\ldots+p_n=p$. So we conclude that $\mu_{2.p}$ is the number partitions of $p$ into a sum of integers $p_k\ge 2$. Thus, one can easily check that
\begin{equation}\label{b1}
	\mu_{2,p}=\nu(p)-\nu(p-1),
\end{equation}
where $\nu(n)$ is the partition function \e{part} (see Fig.\ref{f2}, an example of a double-covering with three connected components).
To characterize $(2,p)^*$-double-coverings \e{x26}, again first consider the connected case, which is possible whenever $p\ge 2$. In contrast to case $(2,p)$-double-coverings one can check that there are only two connected $(2,p)^*$-double-coverings. In Fig.\ref{f3} one can see the graphs of two connected $(2,3)^*$-double-coverings. Thus we can conclude that $\mu_{2,p}^*=2\mu_{2,p}$.
Then, based on \e{part} and \e{b1}, after a simple calculation one can get the following.
\begin{theorem}\label{P1}
	It holds the relation
	\begin{equation}
		\mu_{2,p}^*=2\mu_{2,p}\sim \frac{\pi\sqrt{2}}{12p\sqrt{p}}\exp\left(\pi\left(\frac{2p}{3}\right)^{1/2}\right)\text { as } p\to \infty.
	\end{equation}
\end{theorem}
We investigate the asymptotic behavior of quantities $\mu_{l,p}$ and $\mu_{l,p}^*$ as $p\to \infty$ when $l>2$. 
Based on  \e{b1} and the above mentioned huge formula of Hardy-Ramanujan-Rademacher for $\nu(n)$ one can understand that $\mu_{2,p}$ may not have simple precise formula.  The case of $l>2$ is much more complicated that the case of $l=2$,  and we do not know an exact asymptotic formula for $\mu_{l,p}$ as $l>2$. Instead, we can prove the following bounds.
\begin{theorem}\label{P0}
	If $l,p\ge 2$ are integers and $pl$ is even, then
	\begin{equation}\label{x24}
		(a(l))^{p}\cdot p^{p(l/2-1)}\le \mu_{l,p}(\text{standard})\le \mu_{l,p}^*(\text{full})\le (b(l))^{p}\cdot p^{p(l/2-1)},
	\end{equation}
	where $a(l)$ and $b(l)$ are positive constants depending on $l$.
\end{theorem}
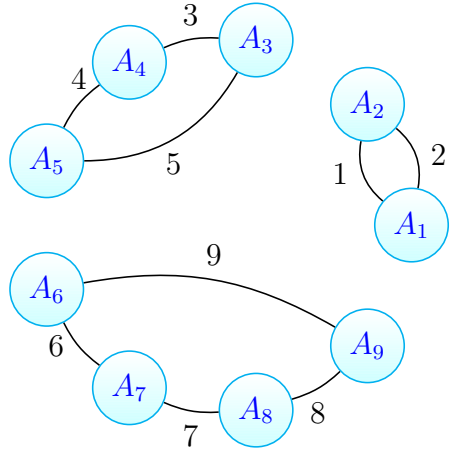
\begin{figure}\label{f2}
	\begin {center}
	\begin{tikzpicture}[semithick ,state/.style ={ circle ,top color =white , bottom color = processblue!20 ,
			draw,processblue , text=blue , minimum width =0.1cm}]
		\node[state] (A) at ( 0:2.5) {$A_1$};
		\node[state] (B) at ( 40:2.5) {$A_2$};
		\node[state] (C) at (2*40:2.5) {$A_3$};
		\node[state] (D) at (3*40:2.5) {$A_4$};
		\node[state] (E) at (4*40:2.5) {$A_5$};
		\node[state] (F) at (5*40:2.5) {$A_6$};
		\node[state] (G) at (6*40:2.5) {$A_7$};
		\node[state] (H) at (7*40:2.5) {$A_8$};
		\node[state] (I) at (8*40:2.5) {$A_9$};
		\path (A) edge  [-]  [bend left =30] node[left  =0.05 cm] {$1$} (B);
		\path (A) edge  [-] [bend right =30] node[right =0.05 cm] {$2$} (B);
		\path (C) edge  [-]   [bend right =13] node[above =0.04 cm] {$3$}(D);
		\path (D) edge  [-]  [bend right =13] node[above =0.05 cm] {$4$}(E);
		\path (C) edge  [-] [bend left =30] node[below =0.05 cm] {$5$}(E);
		\path (F) edge  [-] [bend right =13] node[left =0.05 cm] {$6$}(G);
		\path (G) edge  [-] [bend right =13]node[below =0.05 cm] {$7$}(H);
		\path (H) edge  [-] [bend right =13] node[below =0.05 cm] {$8$}(I);
		\path (F) edge  [-] [bend left =20] node[above =0.05 cm] {$9$}(I);
	\end{tikzpicture}
\end{center}
\caption{a multi-graph with $l=2$, $p=9$.}
\end{figure}
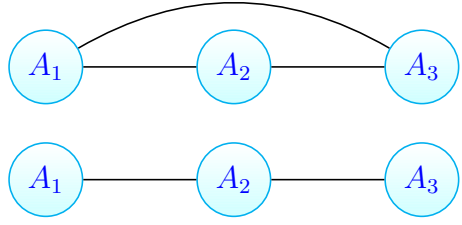
\begin{figure}\label{f3}
	\begin {center}
	\begin {tikzpicture}[-latex ,auto ,node distance =2 cm and 3cm ,on grid ,
	semithick ,
	state/.style ={ circle ,top color =white , bottom color = processblue!20 ,
		draw,processblue , text=blue , minimum width =0.1cm}]
	\node[state] (A){$A_1$};
	\node[state] (B) [right= 2.5cm of A] {$A_2$};
	\node[state] (C) [right= 2.5cm of B] {$A_3$};
	\node[state] (D) [below= 1.5cm of A] {$A_1$};
	\node[state] (E) [right= 2.5cm of D] {$A_2$};
	\node[state] (F) [right= 2.5cm of  E ] {$A_3$};
	\path (B) edge  [-] (C);
	\path (D) edge [-] (E);
	\path (A) edge  [-]  (B);
	\path (E) edge  [-]   (F);
	\path (A) edge  [-] [bend left =30]  (C);
\end{tikzpicture}
\end{center}
\caption{two connected $(2,3)^*$-multi-graphs.}
\end{figure}
Using these estimates, we provide an alternative approach to an inequality involving Rademacher chaos sums proved independently  by Bonami \cite{Bon}  and Kiener \cite{Kie} (see also \cite{Mul} or \cite{Blei}). Let $\{r_n\}_{n\ge 1}$ be the sequence of Rademacher functions $r_n(x)=\sign\left(\sin 2^n\pi x\right)$.
The classical Khintchine inequality states that for any $0<p<\infty$, there are constants $A_p,B_p$ such that
\begin{equation}\label{1}
	A_p\left(\sum_{k=1}^na_k^2\right)^{1/2}\le \left\|\sum_{k=1}^n a_kr_k(x)\right\|_p\le B_p\left(\sum_{k=1}^na_k^2\right)^{1/2}.
\end{equation}
The Khintchine inequality is a well-known object in analysis and probability with various generalizations and applications. A special case of the inequality was first studied by Khintchine \cite{Khi}, proving the right bound of \e{1} with $B_p=\sqrt{p/2+1}$ and $p\ge 2$. Further study of  the inequality were given by Littlewood \cite{Lit}, Paley and Zygmund \cite{PaZy}. 
Let $A_p$ and $B_p$ denote the best constants, for which inequality \e{1} holds. It is trivial that $A_p=1$ if $2\le p<\infty$ and $B_p=1$ for all $0<p\le 2$. In 1961 Stechkin \cite{Ste} proved $B_{p}=((p-1)!!)^{1/p}$ for even integers $p\ge 4$, which then extended for all real numbers $p\ge 3$ by Young \cite{You}. Then solving a long outstanding problem of Littlewood (see. \cite{Lit}) Szarek \cite{Sza} proved that $A_1=1/\sqrt{2}$. In 1982 Haagerup \cite{Haa} introduced a new method, enabling to find the sharp constants in remaining cases and  recovering the prior results newly. He proved that
\begin{equation}
	B_p=2^{1/2}\big(\Gamma(p+1/2)/\pi\big)^{1/p},\quad 2<p<\infty,
\end{equation}
and
\begin{align}
	A_p=\left\{\begin{array}{lrl}
		2^{1/2-1/p}&\text{ if }&0<p<p_0,\\
		2^{1/2}\big(\Gamma(p+1/2)/\sqrt{\pi}\big)^{1/p}&\text{ if }&p>p_0,
	\end{array}
	\right.
\end{align}
where $1<p_0<2$ is the solution of the equation $\Gamma(p+1/2)=\sqrt{\pi}/2$.

Now consider orthonormal systems generated by the products of fixed number of Rademacher functions. For a finite set $A\subset \ZN$ we denote
\begin{equation*}
	w_A(x)=\prod_{k\in A}r_k(x).
\end{equation*}
Recall that such products generate the classical orthonormal system of Walsh functions $W=\{w_A:\, A\subset \ZN\}$. Given integer $l\ge 2$ we denote $\zZ_l=\{A\subset \ZN:\, \#(A)=l\}$ and consider a subsystem $W_l=\{w_A:\, A\in \zZ_l\}\subset W$  of Walsh functions formed by the $l$-wise products of the Rademacher functions. In some literature the system $W_l$ is called a Rademacher chaos. A multiple numerical sequence $b=\{b_A:\, A\in \zZ_l\}$ with the norm $\|b\|=\big(\sum |b_A|^2\big)^{1/2}$ is said to be  finite if there are only finite number of non-zero terms in it. We prove the following bound for the Rademacher chaos sums.
\begin{theorem}\label{T1}
	If $l,p\ge 2$ are integers and $p$ is even, then for any finite sequence $\{b_A:\, A\in \zZ_l\}$ it holds 
	\begin{equation}\label{x33}
		\frac{\big(p!\mu_{l,p}(\text{standard})\big)^{1/p}}{2\sqrt{3}}\le \sup_{\|b\|\le 1}\left\|\sum_{A\in \zZ_l}	b_Aw_A\right\|_p\le \sqrt{l!}\cdot \big(p!\mu_{l,p}(\text{full})\big)^{1/p}.
	\end{equation}
\end{theorem}
From Theorems \ref{P0} and \ref{T1} it immediately follows the Bonami-Kiener hypercontraction inequality, which is an extension of the Kintchine inequality for the Rademacher chaos sums.
\begin{corollary}[Bonami-Kiener, \cite{Bon,Kie}]\label{C1}
	Let $l\ge 2$ be an integer and $p> 2$.  Then for any finite sequence $b=\{b_A:\,A\in \zZ_l\}$, it holds the bound 
	\begin{equation}\label{x1}
		c_l\cdot p^{l/2}\le \sup_{\|b\|\le 1}	\left\|\sum_{A\in \zZ_l}b_Aw_A\right\|_p\le C_l\cdot p^{l/2},
	\end{equation}
	where $C_l>c_l>0$ are constants, depending only on $l$. 
\end{corollary}

\section{Proof of \trm{P0}}\label{S2}

We will need several lemmas. The notation $a\lesssim b$ in this section will stand for the inequality $a\le (c(l))^{p}\cdot b$, where $c(l)$ is a constant depending only on the integer $l$. Having $a\lesssim b$ and $b\lesssim a$ simultaneously, we will write $a\asymp b$. We let $\ZN_k=\{1,2,\ldots, k\}$ and denote by $\pi_k$ the family of all one-to-one mappings (rearrangements) $\sigma:\ZN_k\to \ZN_k$. In case of even $lp$ denote by $\ZD_{l,p}$ and $\ZD^*_{l,p}$ the families of double coverings $\{A_1,\ldots,A_{p}\}$, $A_j\subset \ZN_{pl/2}$, satisfying $\#(A_j)=  l$ or $\#(A_j)\le   l$ respectively. 
It is clear that each double-covering \e{x26} has its isomorphic set-collection in $\ZD_{l,p}^*$. So we can write $\mu_{l,p}^*=\#([\ZD_{l,p}^*])$. We will consider different subfamilies of $\ZD_{l,p}$, indicating the nature of the set-collections of the subfamily. For example the subfamily of the standard set-collections of $\ZD_{l,p}$ will be denoted by $\ZD_{l,p}(\text{standard})$, writing also $\mu_{l,p}(\text{standard})=\#([\ZD_{l,k}(\text{standard})])$. For the full family $\ZD_{l,p}$ sometimes we will also use the notation $\ZD_{l,p}(\text{full})$.
\begin{lemma}\label{L0}
	If $lp$ is even, then we have
	\begin{equation}\label{a24}
		p^{p(l-1)}\lesssim\#(\ZD_{l,p}(\text{full}))\le \#(\ZD_{l,p}^*(\text{full}))\lesssim  p^{p(l-1)}.
	\end{equation}
\end{lemma}
\begin{proof}
	Along with $\ZD_{l,p}^*$ ($\bar \ZD_{l,p}$)  consider the family $\bar \ZD_{l,p}^*$ ($\bar \ZD_{l,p}$) of ordered-collections  $(A_1,\ldots,A_p)$ such that $\{A_1,\ldots,A_p\}\in \ZD_{l,p}^*$ ($\bar \ZD_{l,p}$) . Given double-covering $\ZA=\{A_1,\ldots,A_p\}\in \ZD_{l,p}$ and a rearrangement $\sigma\in \pi_p$ generate another ordered-collection
	\begin{equation}\label{x22}
		(A_{\sigma{(1)}},\ldots,A_{\sigma(p)})\in \bar \ZD_{l,p}^*.
	\end{equation}
	It is clear that the number of different such ordered-collections has an upper bound $p!$ and this bound is reached when all the sets $A_j$ are different. On the other hand let us see that such rearrangements $\sigma$ can generate at least $p!/2^{p/2}$ different ordered-collections \e{x22} ($p$ may also be odd). Indeed, since $\ZA$ is a double-covering it is not possible an equality of three elements $A_j$. So one can see that such a lower bound is reached when each $A_j$ has an equal pair say $A_{j'}$, i.e. $A_j=A_{j'}$. So the number of rearrangements generating the same ordered-collections in \e{a24} can not exceed $2^{p/2}$. This implies that there are at least $p!/2^{p/2}$ different ordered-collections that can be obtained in \e{a24}. Thus, we conclude
	\begin{equation}\label{x23}
		\frac{ \#(\bar \ZD_{l,p}^*)}{p!}\le \#(\ZD_{l,p}^*)\le \frac{ 2^{p/2}\#(\bar \ZD_{l,p}^*)}{p!},
	\end{equation}
	and the same inequality also for $\bar \ZD_{l,p}$. First let us prove the upper bound in \e{x24}. Consider the Cartesian product $G_{l,p}=\ZN_2\times  \ZN_{pl/2}$. For fixed integers $1\le l_j\le l$ randomly chose pairwise disjoint sets $B_j\subset G_{l,p}$, $j=1,2,\ldots,p$, such that  $\#(B_j)=l_j$. Clearly the number of such choices of ordered-collections  $(B_1,\ldots,B_p)$ is equal
	\begin{equation}
		\binom{pl}{l_1}\cdot 	\binom{pl-l_1}{l_2}\ldots \binom{pl-l_1-\ldots-l_{p-1}}{l_p}= \frac{(pl)!}{(l!)^{p}}\lesssim p^{pl}.\label{a23}
	\end{equation}
	Let $(B_1,\ldots,B_p)$ be one of such a choice and suppose $A_j\subset \ZN_{pl/2}$ is the projection of $B_j$, $j=1,2,\ldots,p$, on $\ZN_{pl/2}$. One can check that among with others in this way we obtain all ordered-collections $(A_1,\ldots,A_p)\in \bar\ZD_{l,p}^*$ with $\#(B_j)=l_j$. On the other hand the number of choices $1\le l_j\le $ is equal $l^p$. Thus we conclude $\#(\bar \ZD_{ l,p}^*)\lesssim p^{pl}$.  Combining this and the right inequality of \e{x23}, we get the upper bound of \e{a24}.
	
	To prove the lower bound of \e{a24}, we first suppose that $p$ is even. Then we consider only the ordered-collection $(A_1,\ldots,A_{p})\in \bar \ZD_{l,p}$, where  $\{A_1,\ldots,A_{p/2}\}$ and $\{A_{p/2+1},\ldots, A_{p}\}$ separately form partitions of $\ZN_{pl/2}$. The number of such ordered-collections $(A_1,\ldots,A_{p})$ is equal  $\left(\frac{(pl/2)!}{(l!)^{p/2}}\right)^2$. So we have the bound
	\begin{equation*}
		\#(\bar \ZD_{l,p})\ge \left(\frac{(pl/2)!}{(l!)^{p/2}}\right)^2\gtrsim p^{pl},
	\end{equation*}
	which together with \e{x23} yields the lower bound of \e{a24}. If $p$ is odd and $p=2t+1$, then $l$ must be even. Consider the collections $(A_1,\ldots,A_{p})\in \bar \ZD_{l,p}$, where  each of two set-collections $\{A_1,\ldots,A_{t}\}$ and $\{A_{t+1},\ldots, A_{2t}\}$ consists of pairwise disjoint sets from $\ZN_{pl/2}$, the sets
	\begin{equation}
		A_{2t+1}\cap\left(\cup_{k=1}^{t}A_k\right),\quad A_{2t+1}\cap\left(\cup_{k=t+1}^{2t}A_k\right)
	\end{equation}
	are pairwise disjoint and both have $l/2$ elements. Using a similar argument one can calculate that the number of such ordered-collections is equal 
	\begin{equation*}
		\binom{pl/2}{l}\binom{l}{l/2}\left(\frac{(tl)!}{(l!)^t}\right)^2\gtrsim p^{pl}.
	\end{equation*} 
	Thus we similarly obtain $\#(\bar \ZD_{l,p})\gtrsim p^{pl}$. So, using the left inequality in \e{x23}, we get lower bound of \e{a24}.
\end{proof}

\begin{lemma}\label{L4}
	If $\ZA=\{A_1,A_2,\ldots,A_{p}\}\in \ZD_{l,p}^*(\text{standard})$, then the number of rearrangements $\sigma\in \pi_{pl/2}$, satisfying 
	\begin{equation}\label{a51}
		\{\sigma(A_1),\sigma(A_2),\ldots,\sigma(A_{p})\}=\{A_1,A_2,\ldots,A_{p}\},
	\end{equation}
	doesn't exceed $(3l!)^{p/2}$. 
\end{lemma}
\begin{proof}
	First we suppose that $\ZA=\{A_1,A_2,\ldots,A_{p}\}\in \ZD_{l,p}^*$ is connected. Using the connectivity, one can give a special  numeration $B_1,B_2,\ldots,B_{p}$ of the elements of $\ZA$  such that
	\begin{equation}
		B_{k+1}\cap\left(\cup_{j=1}^kB_j\right)\neq\varnothing,\quad k=1,2\ldots,p-1.
	\end{equation}
	Then we denote 
	\begin{align}
		&B_1^*=B_1,\quad B^*_k=B_k\setminus \cup_{j=1}^{k-1}B_j,\quad k\ge 2,\\
		&l_k=\#(B_k^*).
	\end{align} 
	Of coarse, some of $l_k$'s can be zero and clearly we have 
	\begin{equation*}
		\sum_{k=1}^pl_k=pl/2.
	\end{equation*}
	Let $\sigma$ be a randomly chosen rearrangement, satisfying \e{a51}. Observe that it can be realized on $B_1$ at most in $p\cdot (l)!$ different possible ways. Indeed, for the $\sigma$-image of $B_1$ we have $p$ different choices, since according to \e{a51} it must coincide with one of the sets $A_1,A_2,\ldots,A_p$. Having determined the image of $B_1$, the $\sigma$ on the elements of $B_1$ can be realized at most in $l!$ different ways. So there are at most $p\cdot (l)!$ different realization of $\sigma$ on $B_1$. Suppose by induction the $\sigma$ has been already realized on $\cup_{k=1}^mB_k$. Clearly, with this we will have determined the $\sigma$-image of $B_{m+1}$ as one of the elements of $\ZA$, and  the values of $\sigma$ on each element of the set $B_{m+1}\cap\left(\cup_{k=1}^mB_k\right)$. So to finalize the determination of $\sigma$ on $B_{m+1}$ it remains only to realize $\sigma$ on $B_{m+1}^*$, for which we will have at most $(l_{m+1})!$ different choices.  Thus we conclude that the number of all possible realization of $\sigma$ satisfying \e{a51} totally does not exceed 
	\begin{equation}\label{x20}
		p\cdot \prod_{k=1}^p (l_k)!\le p\cdot (l!)^{p/2},
	\end{equation}
	we just need to show the last inequality. For this end we will consequently use the fact that any product $t!s!$ with $0\le s\le t\le l$ doesn't exceed  $u!v!$ where $0\le u\le v\le l$, $u+v=t+s$ and one of $u$ or $v$ is either $0$ or $l$. If $p$ is even, using the mentioned remark, one can see that the biggest possible value on the left of \e{x20} is reached when each $l_k$ is either $0$ or $l$. So we will have $p/2$ number of $l!$ and $p/2$ number of $0!=1$.  Hence we obtain \e{x20}. If $p$ is odd and so $l$ is even, then the biggest value of \e{x20} is reached when we have $(p-1)/2$ number of $l!$, $(p-1)/2$ number of $0!$ and a single $(l/2)!$. Thus we again get
	\begin{equation}
		p\cdot \prod_{k=1}^p (l_k)!\le p(l!)^{(p-1)/2}(l/2)!\le p\cdot (l!)^{p/2}.
	\end{equation}
	Now suppose that $\ZA=\{\ZA_1,\ldots,\ZA_s\}$ is a general standard set-collection, where $\ZA_k$ are the connectivity classes of $\ZA$. By the definition all E-classes $[\ZA_k]$ are different. Thus any $\sigma$ satisfying \e{a51} acts inside of each connectivity component $\ZA_k$, that is $\sigma(\ZA_k)=\ZA_k$. We also have $\#(\ZA_k)=p_k\ge 2$, $p_1+p_2+\ldots+p_n=p$. Thus applying estimate \e{x20} for each connectivity component, we get the following upper bound for the number of such rearrangements \e{a51}. Namely,
	\begin{equation*}
		\prod_{k=1}^np_k(l!)^{p_k/2}=	(l!)^{p/2}\prod_{k=1}^np_k.
	\end{equation*}
	On the other hand, applying the Cauchy inequality, we have
	\begin{equation*}
		\prod_{k=1}^np_k\le \left(\frac{\sum_{k=1}^np_k}{n}\right)^n=\left(\frac{p}{n}\right)^n.
	\end{equation*}
	Since $1\le n\le p/2$ and the function $(p/x)^x$ takes its maximum on $(0,p/2]$ when $x=p/e$, we conclude $\left(\frac{p}{n}\right)^n\le e^{p/e}\le 3^{p/2}$. Thus we obtain
	\begin{equation*}
		\prod_{k=1}^np_k(l!)^{p_k/2}\le (3l!)^{p/2},
	\end{equation*}
	which finishes the proof of lemma.
\end{proof}

\begin{remark}\label{R1}
	In the sequel for a $\ZA=\{A_1,\ldots,A_{p}\}\in \ZD_{l,p}^*$ the notation $[\ZA]$ will stand for the E-class of $\ZA$ inside $\ZD_{l,p}^*$. Recall that if $\ZA=\{A_1,\ldots,A_{p}\}$ and $\ZB=\{B_1,\ldots,B_p\}\in \ZD_{l,p}^*$ are isomorphic, then there is a rearrangement $\sigma\in \pi_{pl/2}$ such that
	\begin{equation*}
			\{\sigma(A_1),\ldots,\sigma(A_{p})\}=\{B_1,\ldots,B_p\}.
	\end{equation*}
	This implies $\#[\ZA]\le (pl/2)!$ for any $\ZA\in  \ZD_{l,p}$.  Besides, applying  \lem{L4}, for $\ZA\in \ZD_{l,p}^*(\text{standard})$ we have $\#[\ZA]\ge (pl/2)!/(3l!)^{p/2}$.
\end{remark}
Let $p,l\ge 2$ and $\ZA\in \ZD_{l,p}$. A set-collection $\ZA'=\{A_1, A_2,\ldots, A_{r}\}\subset \ZA$, $1\le r\le p$, is said to be X-chain in $\ZA$ if 
\begin{align*}
	&\quad r \text{ is even and }\\
	&\#(A_k\cap A_{k+1})
	=\left\{\begin{array}{rrl}
		l-1&\text{ if }& k \text{ is odd,}\\
		1&\text{ if }& k \text{ is even},
	\end{array}
	\right.
\end{align*}
$k=1,2,\ldots,r-1$.  $\ZA'$ is said to be Y-chain if 
\begin{align*}
	&\quad l \text{ is even and }\\
	&\#(A_k\cap A_{k+1})=l/2,\quad k=1,2,\ldots,r-1.
\end{align*}	
If $l$ is odd, then any single element set-collection $\ZA'=\{A\}$ will also be considered as a Y-chain.
\begin{lemma}\label{L7}
	Let $p,l\ge 2$, $1\le r\le p$ be integers and $\ZA\in \ZD_{l,p}$. Then  there are at most $p$ different X-chains (Y-chains) of length $r$ in $\ZA$.
\end{lemma}
\begin{proof}
	Consider all maximal X-chains (Y-chains) of $\ZA$. Clearly, they do not have a common set of the collection $\ZA$. Any X-chain (Y-chain) of length $r$ must be included in one of such a maximal chain. All this easily imply that the number of X-chains (Y-chains) of length $r$ in $\ZA$ do not exceed $p$. 
\end{proof}
\begin{lemma}\label{L6}
	Let $p>q\ge 2$ and $l\ge 2$ be integers such that both $pl$ and $ql$ are even. Then there is a mapping 
	\begin{equation}
		\phi: \ZD_{l,q}(\text{connected})\to \ZD_{l,p}(\text{connected}),
	\end{equation} 
	satisfying $\#(\phi^{-1}(\ZB))\le p$ for any $\ZB\in \ZD_{l,p}(\text{connected})$. 
\end{lemma}
\begin{proof} 
First suppose that $l$ is odd and so $p-q$ is even. 
	Let $\ZA=\{A_1, A_2,\ldots, A_q\}\in \ZD_{l,q}(\text{connected})$. Chose $a\in A_q$ and $b\in \ZN_{pl/2}\setminus \cup_{k=1}^qA_k$ arbitrarily and define 
	\begin{equation}
		\bar A_q=(A_q\setminus \{a\})\cup \{b\}.
	\end{equation}
Since $p-q$ is even, one can define a X-chain $\{A_{q+1},\ldots,A_p\}$ such that $a\in A_{q+1}$, $b\in A_p$ and
	\begin{equation}\label{x29}
		\ZB=\{A_1,\ldots, \bar A_q,A_{q+1},\ldots,A_p\}\in \ZD_{l,p}(\text{connected}). 
	\end{equation}
Then we define $\phi(\ZA)=\ZB$. Perhaps, $\ZB$ can be the image of different double-coverings from $\ZD_{l,q}(\text{connected})$. One can see that the number of such double-coverings can not exceed the number of  X-chains of $p-q$ length in $\ZB$. So according to \lem{L7} we will always have $\#(\phi^{-1}(\ZB))\le p$.
	
	In the case of even $l$ we treat analogously. At this time we chose arbitrary subsets $A\subset A_q$ and $B\subset \ZN_{pl/2}\setminus \cup_{k=1}^qA_k$ with $\#(A)=\#(B)=l/2$ and define 
	\begin{equation*}
		\bar A_q=(A_q\setminus A)\cup B.
	\end{equation*}
Then we define a $l/2$-chain $\{A_{q+1},\ldots,A_p\}$ with $A\subset  A_{q+1}$, $B\subset  A_p$ satisfying \e{x29}.
The rest of the proof is the same as in the previous case.
\end{proof}

\begin{proof}[Proof of \trm{P0}]
We can suppose that $l>2$. To prove the left inequality in \e{x24}, first, let us first show that 
	\begin{equation}\label{b2}
		\mu_{l,p}(\text{full})\gtrsim p^{p(l/2-1)}.
	\end{equation}
	For any $\ZA\in \ZD_{l,p}(\text{full})$ we have $\#[\ZA]\le (pl/2)!$ (see \rem{R1}). So, using also \lem{L0}, we obtain
	\begin{equation*}
		\mu_{l,p}(\text{full})\ge \frac{\#(\ZD_{l,p}(\text{full}))}{(pl/2)!}\gtrsim \frac{p^{p(l-1)}}{(pl/2)!}\gtrsim p^{p(l/2-1)}.
	\end{equation*}
So it is enough to show that
	\begin{equation}\label{x21}
		\mu_{l,p}(\text{full})=\#[\ZD_{l,p}(\text{full})]\lesssim \#[\ZD_{l,p}(\text{standard})].
	\end{equation}
	Let $\ZA=\{\ZA_1,\ZA_2,\ldots,\ZA_s\}\in \ZD_{l,p}(\text{full})$, where $\ZA_k$ are the connected components of $\ZA$ and some of those can be isomorphic. If a group of components are isomorphic, then we remove all components from the group, keeping one of them. Acting so with all the groups of isomorphic components we will get a standard double-covering  $\ZB=\{\ZB_1,\ldots,\ZB_m\}$ with less connected components than $\ZA$ had. We call $\ZB$ the $\phi$-transformation of $\ZA$, which is uniquely determined up to isomorphism. Observe that $\ZB$ can be the $\phi$-transformation of at most $p\cdot 2^p$ isomorphically different elements of $\ZD_{l,p}(\text{non-standard})$, namely,
	\begin{equation}\label{x39}
		\#[\phi^{-1}(\ZB)]\le p\cdot 2^p.
	\end{equation}
Indeed, let  $\ZB=\{\ZB_1,\ldots,\ZB_m\}$ be the $\phi$-transformation of a $\ZA\in \ZD_{l,p}(\text{non-standard})$. According to the definition of $\phi$-transformation, $\ZA$ should consist of connectivity components each of which is isomorphic to one of the connectivity components of $\ZB$. So we suppose that $\ZA$ has $p_k$ components isomorphic to $\ZB_k$. Obviously we have $p_1+\ldots+p_m\le p$. One can check that the number of possible ordered-collections $(p_1,\ldots,p_m)$ of natural numbers, satisfying this inequality doesn't exceed $p\cdot 2^p$. This yields \e{x39}.

Apply another transformation to $\ZB$ as follows. Chose a connected component in $\ZB$ with a maximal number of elements. Suppose such a component is $\ZB_1$. Applying \lem{L6}, we can replace $\ZB_1$ by a larger connected double-covering $\ZB'_1$ such that $\ZB'=\{\ZB'_1,\ZB_2\ldots, \ZB_m\}\in \ZD_{l,p}(\text{standard})$ and $\ZB'_1,\ZB_2\ldots, \ZB_m$ are the connected components of $\ZB'$. Call $\ZB'$ the $\psi$-transformation of $\ZB$. According to \lem{L6}, one can easily see that $\ZB'$ can be the $\psi$-transformation of at most $p$ double-coverings. Thus, combining also \e{x39}, the superposition $\tau=\psi\circ \phi$ defines a mapping from $\ZD_{l,p}(\text{full})$ into $\ZD_{l,p}(\text{standard})$ such that $\#[\tau^{-1}(\ZB')]\le p^2\cdot 2^p$. This proves 
\begin{equation*}
\#[\ZD_{l,p}(\text{full})]\le p^22^p \#[\ZD_{l,p}(\text{standard})]
\end{equation*}
that implies \e{x21}. Combining also \e{b2}, so we get the left inequality in \e{x24}.
	
		To prove the right inequality of \e{x24}, recall (see \rem{R1}), that for any $\ZA\in \ZD_{l,p}^*(\text{connected})$ we have $\#[\ZA]\ge(pl/2)!/(3l!)^{p/2}$. So, applying \lem{L0}, we conclude
	\begin{equation}\label{a53}
		\#[\ZD_{l,p}^*(\text{connected})]\le 	\frac{(3l!)^{p/2}\cdot \#(\ZD_{l,p}^*(\text{connected}))}{(pl/2)!}\le (d(l))^{p} (p!)^{l/2-1},
	\end{equation}
	where $d(l)$ is a constant depending on $l$ and we can suppose that 
	\begin{equation}\label{x25}
		d(l)>\max_{2\le p\le 100}\#[\ZD_{l,p}^*(\text{full})].
	\end{equation}
	Applying induction on $p$,  we will prove the bound
	\begin{equation}\label{x31}
		\#[\ZD_{l,p}^*(\text{full})]\le 2(d(l))^{p} (p!)^{l/2-1}.
	\end{equation} 
	By \e{x25} it holds if $2\le p\le 100$, so it is enough to consider the case of $p>100$. Hence suppose $q>100$ and we have already proved the estimate for all $p<q$. To prove it for $q$ we estimate the number of connected and non-connected double-coverings separately. For the connected coverings we have the estimate \e{a53}.
Then, using the induction assumption, we obtain
	\begin{align*}
		\#[\ZD_{l,q}^*(\text{non-connected)}]&\le \sum_{2\le k\le q/2}[\ZD_{l,k}^*(\text{full)}]\cdot [\ZD_{l,q-k}^*(\text{full})]\\
		&\le 4(d(l))^{q}\sum_{2\le k\le q/2} (k!(q-k)!)^{l/2-1}.
	\end{align*}
	Clearly we have $k!(q-k)!\le 6(q-3)!$ if $3\le k\le q/2$. Thus we get
	\begin{align}
		&[\ZD_{l,q}^*(\text{non-connected)}]\\
		&\qquad\le 4(d(l))^{q}\left((2(q-2)!)^{l/2-1}+\sum_{3\le k\le q/2} (k!(q-k)!)^{l/2-1} \right)\\
		&\qquad \le 4(d(l))^{q}\left((2(q-2)!)^{l/2-1}+\frac{q}{2}\cdot (6(q-3)!)^{l/2-1}\right)\\
		&\qquad \le (d(l))^{q}(q!)^{l/2-1},\label{a54}
	\end{align}
where the last estimate roughly follows from $q>100$ and $l\ge 3$. From \e{a53} and \e{a54} we obtain \e{x31} for $p=q$. This completes the proof of theorem.
\end{proof}

\section{Proof of \trm{T1}}\label{S3}

Along with estimate \e{x24}, we will need also two more lemmas. The first is a generalization of the classical Cauchy–Schwarz inequality, which is of its own interest. Let $G=\{j_1,j_2,\ldots,j_n\}$ be a set of variables. We call $G$-sequence a finitely supported multiple numerical sequence $b_G=b_{j_1,\ldots, j_n}$, where each $j_k$ take natural values independently. For the sake of further convenience, we denote 
\begin{equation*}
	\sideset{}{'}\sum_{G}\,b_G=\sum_{j_1,j_2,\ldots,j_n\in G}b_{j_1,j_2,\ldots,j_n}.
\end{equation*}
\begin{lemma}\label{L10}
	Let $I_k$, $k=1,2,\ldots, p$, be sets of variables such that $\{I_1,\ldots,I_p\}$ is an even-covering and let $G=\cup_{j=1}^pI_j$. Then for any positive numerical $I_k$-sequences $a^{(k)}_{I_k}$, $k=1,2,\ldots,p$ we have
	\begin{equation}\label{a6}
		\sideset{}{'}\sum_{G}a^{(1)}_{I_1}\ldots a^{(p)}_{I_{p}}\le \prod_{k=1}^{p}\left(\sideset{}{'}\sum_{I_k} \left(a^{(k)}_{I_{k}}\right)^2\right)^{1/2}.
	\end{equation}
\end{lemma}
\begin{proof}
First suppose that $\{I_1,\ldots,I_p\}$ is a double-covering. Apply induction with respect to $n=\#(G)$. If $n=1$ , then \e{a6} gives the classical Cauchy–Schwarz inequality. Assume the inequality holds for all $n< m$ and let us prove it for $n=m$. Hence we let $\#(G)=m$ and variable sets $I_k\subset G$, $k=1,2,\ldots, p$, form a double-covering and $G=\cup_kI_k$. Suppose two of these sets, say $I_1$ and $I_2$, have nonempty intersection, $I_1\cap I_2\neq\varnothing$. According to the double-covering condition we have $I_j\cap(I_1\cap I_2)=\varnothing$ for all $j\ge 3$. Thus, using the classical Cauchy–Schwarz inequality, we can write
	\begin{align}
		&\sideset{}{'}\sum_{G}a^{(1)}_{I_1}\ldots a^{(p)}_{I_p}\\
		&\qquad =\sideset{}{'}\sum_{G\setminus (I_1\cap I_2)}a^{(3)}_{I_3}\ldots a^{(p)}_{I_p}\sideset{}{'}\sum_{I_1\cap I_2}a^{(1)}_{I_1}a^{(2)}_{I_2}\\
		&\qquad\le\sideset{}{'}\sum_{G\setminus (I_1\cap I_2)}a^{(3)}_{I_3}\ldots a^{(p)}_{I_p}\left(\sideset{}{'}\sum_{I_1\cap I_2}|a^{(1)}_{I_1}|^2\right)^{1/2}\left(\sideset{}{'}\sum_{I_1\cap I_2}|a^{(2)}_{I_2}|^2\right)^{1/2}\\	
		&\qquad=\sideset{}{'}\sum_{G\setminus (I_1\cap I_2)}a'_{I_1\setminus I_2}\cdot a''_{I_2\setminus I_1}\cdot a^{(3)}_{I_3}\ldots a^{(p)}_{I_p},\label{a5}\\ 
	\end{align}
	where 
	\begin{equation}\label{x42}
		a'_{I_1\setminus I_2}=\left(\sideset{}{'}\sum_{I_1\cap I_2}|a^{(1)}_{I_1}|^2\right)^{1/2},\quad a''_{I_2\setminus I_1}=\left(\sideset{}{'}\sum_{I_1\cap I_2}|a^{(2)}_{I_2}|^2\right)^{1/2}.
	\end{equation}
	Observe that $\{I_1\setminus I_2, I_2\setminus I_1,I_3,\ldots,I_{p}\}$ is double-covering, whose base is $G'=G\setminus (I_1\cap I_2)$ and $\#(G')<m$.
	So using the induction assumption and notation \e{x42},  we obtain
	\begin{align*}
		\sideset{}{'}\sum_{G\setminus (I_1\cap I_2)}&a'_{I_1\setminus I_2}\cdot a''_{I_2\setminus I_1}\cdot a^{(3)}_{I_3}\ldots a^{(p)}_{I_p}\\
		&\le \left(\sideset{}{'}\sum_{I_1\setminus I_2}|a'_{I_1\setminus I_2}|^2\right)^{1/2} \left(\sideset{}{'}\sum_{I_2\setminus I_1}|a''_{I_2\setminus I_1}|^2\right)^{1/2}\\
		&\qquad\qquad\times\left(\sideset{}{'}\sum_{I_3}|a^{(3)}_{I_3}|^2\right)^{1/2}\ldots
		\left(\sideset{}{'}\sum_{I_{p}}|a^{(p)}_{I_p}|^2\right)^{1/2}=\prod_{k=1}^{p}\left(\sideset{}{'}\sum_{I_k} |a^{(k)}_{I_k}|^2\right)^{1/2}.
	\end{align*} 
Now suppose that $\{I_1,\ldots,I_p\}$ is an arbitrary even covering of variable sets and four of them say $I_1,I_2,I_3,I_4$ have a common element $k$. Chose an arbitrary variable  $k'\not\in \cup_{j=1}^pI_j$ and replace the sets $I_1$ and $I_2$ by $\tilde I_1=(I_1\setminus \{k\})\cup \{k'\}$ and $\tilde I_2=(I_2\setminus \{k\})\cup \{k'\}$ respectively. One can check that
\begin{equation}
	\sideset{}{'}\sum_{G}a^{(1)}_{I_1}a^{(2)}_{I_1}a^{(3)}_{I_3}\ldots a^{(p)}_{I_{p}}\le	\sideset{}{'}\sum_{G\cup \{k'\}}a^{(1)}_{\tilde I_1}a^{(2)}_{\tilde I_2}a^{(3)}_{ I_3}\ldots a^{(p)}_{I_{p}},
\end{equation}
since the first sum is a part of the second sum. Applying this procedure consecutively, finally we will get a dominating sum with a double-covering collection of variable sets. This will reduce the general case of the inequality to the case of double-covering,
completing the proof of lemma.
\end{proof}
\begin{lemma}\label{L1}
	If $\ZA=\{A_1,A_2,\ldots,A_{p}\}$ is an even-covering, then there exists a double-covering $\ZB=\{B_1,B_2,\ldots,B_{p}\}$ and an onto mapping
	\begin{equation*}
		\phi:\bigcup_{k=1}^{p}B_k\to\bigcup_{k=1}^{p}A_k
	\end{equation*}
	such that $\phi(B_k)=A_k$, $k=1,2,\ldots, p$.
\end{lemma}
\begin{proof}
	We define $\ZB$ inductively, "transforming" the sets of $\ZA$ as follows. Suppose there are four sets of $\ZA$ say $A_1,A_2,A_3,A_4$ having a common point $a$. Chose an arbitrary point  $b\not\in \cup_{j=1}^pA_j$, then replace the sets $A_1$ and $A_2$ by $(A_1\setminus \{a\})\cup \{b\}$ and $(A_2\setminus \{a\})\cup \{b\}$ respectively. Then we will say that $b$ is generated from $a$.  Applying this procedure consecutively, we will finally transform $\ZA$ to a double-covering $\ZB$. For $a\in \cup_k A_k$ denote by $[a]$ the union of $a$ with all the points of $\cup_kB_k$ that are generated from $a$. Define the mapping $\phi$, assigning to each point $b\in \cup_kB_k$ to $a$ whenever $b\in [a]$. Obviously it satisfies the requirements of the lemma.
\end{proof}

Define a virtual function $\gamma(a_1,a_2,\ldots,a_p)$, where $a_k$ can be either integers or sets. If all $a_k$ are different, then $\gamma(a_1,a_2,\ldots,a_p)=p!$. Otherwise grouping the equal elements together, we can split the collection $\{a_1,a_2,\ldots,a_p\}$ into groups of $p_k$ elements, such that $p_1+p_2+\ldots+p_m=p$. In that case we define
\begin{equation}\label{x32}
	\gamma(a_1,a_2,\ldots,a_p)=\frac{p!}{(p_1)!(p_2)!\ldots (p_m)!}.
\end{equation}
Observe that if $\ZA=\{A_1,\ldots,A_{p}\}$ is a double-covering, then 
\begin{equation}\label{x44}
	\gamma(A_1,A_2,\ldots,A_{p})\ge \frac{p!}{2^p},
\end{equation}
since at most two sets of $\ZA$ can be equal each other. Here is another well-known use of the function $\gamma$ in the polynomial formula
\begin{equation}\label{x43}
	(x_1+x_2+\ldots+x_n)^p=\sum_{1\le k_1,\ldots,k_p\le n}\gamma(k_1,k_2,\ldots,k_p)x_{k_1}x_{k_2}\ldots x_{k_p}.
\end{equation}
\begin{proof}[Proof pf \trm{P1}]
First we prove the right side inequality in \e{x24}. Let $\{b_A:\, A\in \zZ_l\}$ be a finitely supported sequence. Using \e{x43}, we get
	\begin{align}
		&\int_0^1\left(\sum_{A\in \zZ_l}	b_Aw_A\right)^{p}\nonumber\\
		&\qquad = \sum_{A_1,\ldots,A_{p}\in \zZ_l}\gamma(A_1,\ldots,A_{p})b_{A_1}\ldots b_{A_p}\int_0^1w_{A_1}\ldots w_{A_p}.\label{x30}
	\end{align}
	Observe that if the set-collection $\{A_1,\ldots,A_{p}\}$ is not an even-covering, then the integral in \e{x30} is zero, 
	otherwise it is $1$. Thus we conclude
	\begin{align}
		\int_0^1\left(\sum_{A\in \zZ_l}	b_Aw_A\right)^{p}&=\sum_{\stackrel{A_1,\ldots,A_{p}\in \ZA_{l}}{ \{A_1,\ldots,A_{p}\}\text { is even}}}\gamma(A_1,A_2,\ldots,A_{p})b_{A_1}\ldots b_{A_{p}}\label{a9}\\
		&\le \sum_{\stackrel{A_1,\ldots,A_{p}\in \ZA_{l}}{ \{A_1,\ldots,A_{p}\}\text { is even}}}\gamma(A_1,A_2,\ldots,A_{p})|b_{A_1}\ldots b_{A_{p}}|.\label{x51}
	\end{align}
	From this moment we define $b_{j_1,j_2,\ldots,j_l}=0$ if at least two of $j_k$ are equal. Observe that any term $|b_{A_1}\ldots b_{A_{p}}|$ with an even-covering $\ZA=\{A_1,\ldots,A_{p}\}$ is included at least in one of the prime-sum
	\begin{equation}\label{a8}
		\sideset{}{'}\sum_{\cup_{j=1}^{p}I_j}|b_{I_1}\ldots b_{I_{p}}|,
	\end{equation}
	where the variable set-collection $I=\{I_1,\ldots,I_{p}\}$ forms a double-covering with $\#(I_k)=l$. Indeed, suppose $\{A_1,\ldots,A_{p}\}$ is an even-covering of naturals. According to \lem{L1} there is a variable double-covering $\{I_1,\ldots,I_{p}\}$ and an onto mapping 
	\begin{equation*}
		\phi:\bigcup_{k=1}^{p}I_k\to\bigcup_{k=1}^{p}A_k
	\end{equation*}
	such that $\phi(I_k)=A_k$. So, if we assign to any variable $j\in \cup_kI_k$ in $|b_{I_1}\ldots b_{I_{p}}|$ the value $\phi(j)\in \ZN$, we will get the term $|b_{A_1}\ldots b_{A_{p}}|$ in \e{a8}. This remark and \e{x32} yield 
	\begin{align}
		\sum_{\stackrel{A_1,\ldots,A_{p}\in \ZA_{l}}{ \{A_1,\ldots,A_{p}\}\text { is even}}}\gamma(A_1,\ldots,A_{p})|b_{A_1}\ldots b_{A_p}|\le p! \sum_{[I_1,\ldots,I_{p}]}\sideset{}{'}\sum_{\cup_{j=1}^{p}I_j}|b_{I_1}\ldots b_{I_p}|,\label{a10}
	\end{align}
	where the first summation is taken over all isometrically different variable double-coverings $\{I_1,\ldots,I_{p}\}$ with $\#(I_k)=l$. Recall that the number of such double-coverings is $\mu_{l,p}(\text{full})$. On the other hand, applying \lem{L10} with sequences $a^{(k)}_{I_k}=b_{I_k}$, $k=1,2,\ldots,p$, we obtain 
	\begin{equation}\label{a11}
		\sideset{}{'}\sum_{\cup_{j=1}^{p}I_j}|b_{I_1}\ldots b_{I_p}|\le  \left(\sideset{}{'}\sum_{I_1}(b_{I_1})^2\right)^{p/2}= \left(l!\sum_{A\in \zZ_l}(b_A)^2\right)^{p/2}.
	\end{equation}
	Thus, combining \e{x51}, \e{a10} and \e{a11}, we obtain 
	\begin{equation}
	\left\|\sum_{A\in \zZ_l}	b_Aw_A\right\|_p\le \sqrt{l!}\cdot (p!\mu_{l,p}(\text{full}))^{1/p}\left(\sum_{I\in \zZ_l}|b_I|^2\right)^{1/2},
	\end{equation}
 which gives the right side inequality in \e{x33}.
To prove of the left side inequality consider the finite sum
\begin{equation}\label{x52}
	\sum_{A\in \zZ_l(m)} w_A,
\end{equation}
where
\begin{equation*}
	\zZ_l(m)=\{A\subset \ZN_m=\{1,2,\ldots,m\}:\, \#(A)=l\}\subset \zZ_l. 
\end{equation*}
We have
\begin{align}
	\int_0^1&\left(\sum_{A\in \zZ_l(m)} w_A\right)^{p}\\
	&=\sum_{A_1,\ldots,A_{p}\in \zZ_l(m)}\gamma(A_1,A_2,\ldots,A_{p})\int_0^1w_{A_1}w_{A_2}\ldots w_{A_p}\\
	&=\sum_{\stackrel{A_1,\ldots,A_{p}\in \zZ_{l}(m)}{ \{A_1,\ldots,A_{p}\}\text { is even}}}\gamma(A_1,A_2,\ldots,A_{p})\\
	&\ge \sum_{\stackrel{A_1,\ldots,A_{p}\in \zZ_{l}(m)}{ \{A_1,\ldots,A_{p}\}\text { is standard double-covering}}}\gamma(A_1,A_2,\ldots,A_{p}).\label{x34}
\end{align}
Let $\zI=\{I_1,\ldots,I_{p}\}$ be a standard double-covering of variables such that $\#(I_k)=l$. Realizing the variables of $\cup_{j=1}^pI_k$ on $\ZN_m$ independently, the collection $\zI$ will generate all the standard double-coverings $ \ZA=\{A_1,\ldots,A_{p}\}$ with $A_k\in \zZ_l(m)$ that are isomorphic to $\zI$. According to \lem{L4}, each such $\ZA$ can meet in these realizations at most $(3l!)^{p/2}$ time. 
Thus, applying \e{x44} and continuing estimation \e{x34}, we get
\begin{align}
		\int_0^1\left(\sum_{A\in \zZ_l(m)} w_A\right)^{p}&\ge \frac{1}{(3l!)^{p/2}}\sum_{[I_1,\ldots,I_{p}]}\quad \sideset{}{'}\sum_{\cup_{j=1}^{p}I_j\subset \ZN_m}\gamma(I_1,I_2,\ldots,I_{p})\label{x54}\\
		&\ge \frac{p!}{2^{p}(3l!)^{p/2}}\cdot \mu_{l,p}(\text{standard})A^{pl/2}_m,\label{x53}
\end{align}
where the first summation is taken over all the isomorphically different standard double-covering of variables such that $\#(I_k)=l$. So the quantity $\mu_{l,p}(\text{standard})$ in \e{x53} stands for the number of such coverings, and
\begin{equation*}
	A^{pl/2}_m=m(m-1)\ldots(m-pl/2+1)
\end{equation*}
is the number of the terms in the second sum of \e{x54}. The $l^2$ norm of the coefficients of sum \e{x52} is exactly $\binom{m}{l}^{1/2}$
and we have
\begin{align*}
	\frac{\left\|\sum_{A\in \zZ_l(m)} w_A\right\|_{p}}{\binom{m}{l}^{1/2}}\ge  \frac{\left(p!\cdot \mu_{l,p}(\text{standard})A_m^{pl/2}\right)^{1/p}}{2\sqrt{3l!}\cdot \binom{m}{l}^{1/2}}\to \frac{\left(p!\mu_{l,p}(\text{standard})\right)^{1/p}}{2\sqrt{3}}
\end{align*}
as $m\to\infty$. This implies the left inequality in \e{x33}.
\end{proof}

\section{Open questions}
In this section we state some open problems that are interesting from the combinatorial point of view. Consider quantities
\begin{align}
	&d(l)=\sup_{p\ge 2}\frac{(\mu_{l,p}(\text{full}))^{1/p}}{p^{l/2-1}},\\
	&g(l)=\sup_{p\ge 2}\frac{\sup_{\|b\|\le 1}	\left\|\sum_{A\in \zZ_l}b_Aw_A\right\|_p}{p^{l/2}},
\end{align}
where the first one has a pure combinatorial character, while the second describes the sharp Khintchine constant in \e{x1}. From \e{x1} and \e{x24} it follows that $d(l)<\infty$ and $g(l)<\infty$ for all integers $l\ge 2$. A careful examination of the proofs in \sect{S2} may provide the bound $d(l)\le c2^{l/2}$ with an absolute constant $c>0$. So using \pro{P1},  we can also say that $g(l)\le c\sqrt{2^l\cdot l!}$. Using a standard argument one can deduce from the  Bonami-Kiener inequality \e{x1}  the bound
\begin{equation}\label{d3}
	\left\|\sum_{A\in \zZ_l}b_Aw_A\right\|_p\le C_{p,q,l}\left\|\sum_{A\in \zZ_l}b_Aw_A\right\|_q,
\end{equation} 
for any Rademacher chaos sum, where $1\le q<p<\infty$. In contrast to classical Rademacher case (when $l=1$) to the best of our knowledge the optimal constant that can be in \e{d3}, is not known for any combination of the parameters $p> q$. Some estimates of the optimal constant of \e{d3} one can find in papers \cite{Jan, Don, Lar, Ivan}.
\begin{problem}
Find the exact values of $d(l)$ and $g(l)$.
\end{problem}
From \lem{L6} it follows that if $p>q$ and $pl$ and $ql$ are even, then
\begin{equation*}
	\#(\ZD_{l,q}(\text{connected}))\le p\cdot \#(\ZD_{l,p}(\text{connected})).
\end{equation*}
We do not know whether the factor $p$ on the right side can be removed if $l>2$. Moreover, 
\begin{problem}
Let $l>2$. Prove that for any $p>q$ with even $pl$, $ql$ the following bounds hold:
\begin{align*}
	&\#(\ZD_{l,q}(\text{connected}))\le \#(\ZD_{l,p}(\text{connected})),\\
	&\#(\ZD_{l,q}(\text{full}))\le \#(\ZD_{l,p}(\text{full})).
\end{align*}
\end{problem}
\begin{problem}
	Let $l>2$. Prove that whenever $p>q$ and $pl$, $ql$ are even,  $\mu_{l,q}(\text{full})\le \mu_{l,p}(\text{full})$ and $\mu_{l,q}(\text{connected})\le \mu_{l,p}(\text{connected})$.
\end{problem}
As it was mentioned in the introduction $\mu_{l,p}=\#([\ZD_{l,p}(\text{full})])$ is the number of $p$-vertex unlabeled multigraphs with $l$-degree vertexes. One can also check that $\#(\ZD_{l,p}(\text{full}))$ is the number of the same kind of labeled graphs $\#([\ZD_{l,p}(\text{full})])$. Relations \e{x24} and \e{a24} give restricted characterizations of the asymptotic behaviors of these two quantities.   
\begin{problem}
	For $l>2$ prove that the limits 
	\begin{equation*}
		\lim_{p\to\infty}\frac{(\#([\ZD_{l,p}(\text{full})]))^{1/p}}{p^{l/2-1}}\text { and } \lim_{p\to\infty}\frac{(\#(\ZD_{l,p}(\text{full})))^{1/p}}{p^{l-1}}
	\end{equation*}
exist and find those values.
\end{problem}

\end{document}